\documentclass[11pt,letterpaper]{article}
\usepackage{amssymb,amsmath,graphicx,amsfonts}
\usepackage{amsmath}
\usepackage{amsfonts}
\usepackage{tikz}
\usetikzlibrary{arrows}
\usepackage{color}

\newtheorem{theorem}{Theorem}

\newtheorem{lemma}[theorem]{Lemma}

\newtheorem{corollary}[theorem]{Corollary} 
\newtheorem{remark}[theorem]{Remark}
\newtheorem{example}[theorem]{Example}
\newtheorem{problem}[theorem]{Problem}

\newtheorem{definition}[theorem]{Definition}
\newtheorem{conjecture}[theorem]{Conjecture}
\newtheorem{preproof}{{\bf Proof}}

\newenvironment{proof}[1]{\begin{preproof}{\rm
			#1}\hfill{$\blacksquare$}}{\end{preproof}}
\newtheorem{presproof}{{\bf Sketch of Proof.\ }}

\newtheorem{prepro}{{\bf Proposition}}

\title{Simultaneous coloring of vertices and incidences of graphs}
{\small
	\author{Mahsa Mozafari-Nia$^a$, Moharram N. Iradmusa$^{a,b}$\\
		{\small $^{a}$Department of Mathematical Sciences, Shahid Beheshti University,}\\
		{\small G.C., P.O. Box 19839-63113, Tehran, Iran.}\\
		{\small $^{b}$School of Mathematics, Institute for Research in Fundamental Sciences (IPM),}\\
		{\small P.O. Box: 19395-5746, Tehran, Iran.}}
	
\begin{document}
		\maketitle
		\begin{abstract}
			An $n$-subdivision of a graph $G$ is a graph constructed by replacing a path of length $n$ instead of each edge of $G$ and an $m$-power of $G$ is a graph with the same vertices as $G$ and any two vertices of $G$ at distance at most $m$ are adjacent. The graph $G^{\frac{m}{n}}$ is the $m$-power of the $n$-subdivision of $G$. In [M. N. Iradmusa, M. Mozafari-Nia, A note on coloring of $\frac{3}{3}$-power of subquartic graphs, Vol. 79, No.3, 2021] it was conjectured that the chromatic number of $\frac{3}{3}$-power of graphs with maximum degree $\Delta\geq 2$ is at most $2\Delta+1$. In this paper, we introduce the simultaneous coloring of vertices and incidences of graphs and show that the minimum number of colors for simultaneous proper coloring of vertices and incidences of $G$, denoted by $\chi_{vi}(G)$, is equal to the chromatic number of $G^{\frac{3}{3}}$. Also by determining the exact value or the upper bound for the said parameter, we investigate the correctness of the conjecture for some classes of graphs such as $k$-degenerated graphs, cycles, forests, complete graphs and regular bipartite graphs. In addition, we investigate the relationship between this new chromatic number and the other parameters of graphs.
		\end{abstract}
		\section{Introduction}\label{sec1}
		All graphs we consider in this paper are simple, finite and undirected. For a graph $G$, we denote its vertex set, edge set and face set (if $G$ is planar) by $V(G)$, $E(G)$ and $F(G)$ respectively. Maximum degree, independence Number
		and maximum size of cliques of $G$ are denoted by $\Delta(G)$, $\alpha(G)$ and $\omega(G)$, respectively.  Also, for vertex $v\in V(G)$, $N_G(v)$ is the set of neighbors of $v$ in $G$ and any vertex of degree $k$ is called a $k$-vertex.. From now on,  we use the notation $[n]$ instead of $\{1,\ldots,n\}$. We mention some of the definitions that are referred to throughout the note and for other necessary definitions and notations we refer the reader to a standard text-book \cite{bondy}.\\
		A mapping $c$ from $V(G)$ to $[k]$ is a proper $k$-coloring of $G$, if $c(v)\neq c(u)$ for any two adjacent vertices. A minimum integer $k$ that $G$ has a proper $k$-coloring is the chromatic number of $G$ and denoted by $\chi(G)$. Instead of the vertices, we can color the edges of graph. A mapping $c$ from $E(G)$ to $[k]$ is a proper edge-$k$-coloring of $G$, if $c(e)\neq c(e')$ for any two adjacent edges $e$ and $e'$ ($e\cap e'\neq\varnothing$). A minimum integer $k$ that $G$ has a proper edge-$k$-coloring is the chromatic index of $G$ and denoted by $\chi'(G)$.\\
		Another coloring of graph is the coloring of incidences of graphs. The concepts of incidence, incidence graph and incidence coloring were introduced by Brualdi and Massey in 1993 \cite{Bruldy}. In graph $G$, any pair $i=(v,e)$ is called an incidence of $G$, if $v\in V(G)$, $e\in E(G)$ and $v\in e$. Also in this case the elements $v$ and $i$ are called incident. For any edge $e=\{u,v\}$, we call $(u,e)$, the first incidence of $u$ and $(v,e)$, the second incidence of $u$. In general, for a vertex $v\in V(G)$, the set of the first incidences and the second incidences of $v$ is denoted by $I_1^G(v)$ and $I_2^G(v)$, respectively. Also let $I_G(v)=I_1^G(v)\cup I_2^G(v)$ , $I_1^G[v]=\{v\}\cup I_1^G(v)$ and $I_G[v]=\{v\}\cup I_G(v)$. Sometime we remove the index $G$ for simplicity.\\
		Let $I(G)$ be the set of the incidences of $G$. The incidence graph of $G$, denoted by $\mathcal{I}(G)$, is a graph with vertex set $V(\mathcal{I}(G))=I(G)$ such that two incidences $(v,e)$ and $(w,f)$ are adjacent in $\mathcal{I}(G)$ if $(i)$ $v=w$, or $(ii)$ $e=f$, or $(iii)$ $\{v,w\}=e$ or $f$. Any proper $k$-coloring of $\mathcal{I}(G)$ is an incidence $k$-coloring of $G$. The incidence chromatic number of $G$, denoted by $\chi_i(G)$, is the minimum integer $k$ such that $G$ is incidence $k$-colorable.\\
		Total coloring is one of the first simultaneous colorings of graphs. A mapping $c$ from $V(G)\cup E(G)$ to $[k]$ is a proper total-$k$-coloring of $G$, if $c(x)\neq c(y)$ for any two adjacent or incident elements $x$ and $y$. A minimum integer $k$ that $G$ has a proper total-$k$-coloring is the total chromatic number of $G$ and denoted by $\chi''G)$ \cite{behzad}. In 1965, Behzad conjectured that  $\chi''(G)$ never exceeds $\Delta(G)+2$.\\
		Another simultaneous coloring began in the mid-1960s with Ringel \cite{ringel}, who conjectured that the vertices and faces of a planar graph may be colored with six colors such that every two adjacent or incident of them are colored differently. In addition to total coloring which is defined for any graph, there are three other types of simultaneous colorings of a planar graph $G$, depending on the use of at least two sets of the sets $V(G)$, $E(G)$, and $F(G)$ in the coloring. These colorings of graphs have been studied extensively in the literature and there are many results and also many open problems. For further information see \cite{borodin, chan, wang1,wang2}.\\
		Inspired by the total coloring of a graph $G$ and its connection with the fractional power of graphs which was introduced in \cite{paper13}, in this paper we define a new kind of simultaneous coloring of graphs. In this type of coloring, we color simultaneously the vertices and the incidences of a graph.
		\begin{definition}\label{verinccol}
			Let $G$ be a graph. A $vi$-simultaneous proper $k$-coloring of $G$ is a coloring $c:V(G)\cup I(G)\longrightarrow[k]$ in which any two adjacent or incident elements in the set $V(G)\cup I(G)$ receive distinct colors. The $vi$-simultaneous chromatic number, denoted by $\chi_{vi}(G)$, is the smallest integer k such that $G$ has a $vi$-simultaneous proper $k$-coloring.
		\end{definition}
		\begin{example}
			{\rm Suppose cycles of order 3 and 4. we know that $\chi(C_3)=\chi'(C_3)=3$ and $\chi''(C_3)=\chi_i(C_3)=4$. But four colors are not enough for $vi$-simultaneous proper coloring of $C_3$ and easily one can show that $\chi_{vi}(C_3)=5$. For the cycle of order four, we have $\chi(C_4)=\chi'(C_4)=2$ and $\chi''(C_4)=\chi_i(C_4)=4$. In addition, Figure \ref{C4} shows that $\chi_{vi}(C_4)=4$.}
		\end{example}
		\begin{figure}[h]
			\begin{center}
					\begin{tikzpicture}[scale=1.0]
						\tikzset{vertex/.style = {shape=circle,draw, line width=1pt, opacity=1.0, inner sep=2pt}}
						\tikzset{vertex1/.style = {shape=circle,draw, fill=black, line width=1pt,opacity=1.0, inner sep=2pt}}
						\tikzset{arc/.style = {->,> = latex', line width=1pt,opacity=1.0}}
						\tikzset{edge/.style = {-,> = latex', line width=1pt,opacity=1.0}}
						\node[vertex1] (a) at (0,0) {};
						\node  at (-0.3,-0.3) {$1$};
						\node[vertex] (b) at (1,0) {};
						\node  at (1,-0.4) {$2$};
						\node[vertex] (c) at  (2,0) {};
						\node  at (2,-0.4) {$3$};
						\node[vertex1] (d) at  (3,0) {};
						\node  at (3.3,-0.3) {$4$};
						\node[vertex] (e) at  (3,1) {};
						\node  at (3.4,1) {$1$};
						\node[vertex] (f) at  (3,2) {};
						\node  at (3.4,2) {$2$};
						\node[vertex1] (g) at (3,3) {};
						\node  at (3.3,3.3) {$3$};
						\node[vertex] (h) at (2,3) {};
						\node  at (2,3.4) {$4$};
						\node[vertex] (i) at (1,3) {};
						\node  at (1,3.4) {$1$};
						\node[vertex1] (j) at (0,3) {};
						\node  at (-0.3,3.3) {$2$};
						\node[vertex] (k) at (0,2) {};
						\node  at (-0.4,2) {$3$};
						\node[vertex] (m) at (0,1) {};
						\node  at (-0.4,1) {$4$};
						\draw[edge] (a)  to (b);
						\draw[edge] (b)  to (c);
						\draw[edge] (c)  to (d);
						\draw[edge] (d)  to (e);
						\draw[edge] (e)  to (f);
						\draw[edge] (f)  to (g);
						\draw[edge] (g)  to (h);
						\draw[edge] (h)  to (i);
						\draw[edge] (i)  to (j);
						\draw[edge] (j)  to (k);
						\draw[edge] (k)  to (m);
						\draw[edge] (m)  to (a);
						\node[vertex1] (a1) at (5,0) {};
						\node  at (4.7,-0.3) {$a$};
						\node[vertex] (b1) at (6,0) {};
						\node  at (6,-0.4) {$(a,b)$};
						\node[vertex] (c1) at  (7,0) {};
						\node  at (7,-0.4) {$(b,a)$};
						\node[vertex1] (d1) at  (8,0) {};
						\node  at (8.3,-0.3) {$b$};
						\node[vertex] (e1) at  (8,1) {};
						\node  at (8.6,1) {$(b,c)$};
						\node[vertex] (f1) at  (8,2) {};
						\node  at (8.6,2) {$(c,b)$};
						\node[vertex1] (g1) at (8,3) {};
						\node  at (8.3,3.3) {$c$};
						\node[vertex] (h1) at (7,3) {};
						\node  at (7,3.4) {$(c,d)$};
						\node[vertex] (i1) at (6,3) {};
						\node  at (6,3.4) {$(d,c)$};
						\node[vertex1] (j1) at (5,3) {};
						\node  at (4.7,3.3) {$d$};
						\node[vertex] (k1) at (5,2) {};
						\node  at (4.4,2) {$(d,a)$};
						\node[vertex] (m1) at (5,1) {};
						\node  at (4.4,1) {$(a,d)$};
						\draw[edge] (a1)  to (b1);
						\draw[edge] (b1)  to (c1);
						\draw[edge] (c1)  to (d1);
						\draw[edge] (d1)  to (e1);
						\draw[edge] (e1)  to (f1);
						\draw[edge] (f1)  to (g1);
						\draw[edge] (g1)  to (h1);
						\draw[edge] (h1)  to (i1);
						\draw[edge] (i1)  to (j1);
						\draw[edge] (j1)  to (k1);
						\draw[edge] (k1)  to (m1);
						\draw[edge] (m1)  to (a1);
					\end{tikzpicture}
				\caption{$vi$-simultaneous proper $4$-coloring of $C_4$. Black vertices are corresponding to the vertices of $G$ and white vertices are corresponding to the incidences of $C_4$. The incidence $(u,\{u,v\})$ is denoted by $(u,v)$.}
				\label{C4}
			\end{center}
		\end{figure}
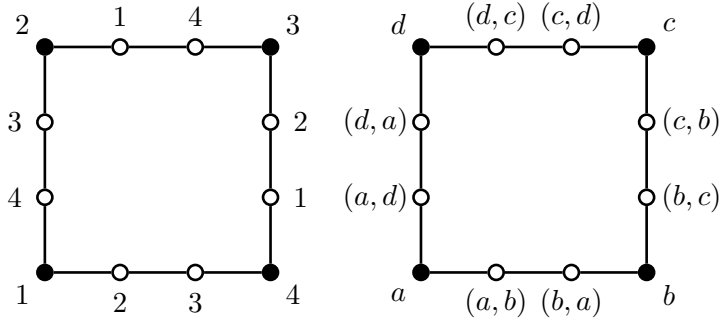
		Similar to incidence coloring, we can define some special kind of $vi$-simultaneous coloring of graphs according to the number of colors that appear on the incidences of each vertex.
		\begin{definition}\label{(k,l)IncidenceCol}
			A $vi$-simultaneous proper $k$-coloring of a graph $G$ is called $vi$-simultaneous $(k,s)$-coloring of $G$ if for any vertex $v$, the number of colors used for coloring $I_2(v)$ is at most $s$. We denote by $\chi_{vi,s}(G)$ the smallest number of colors required for a $vi$-simultaneous $(k,s)$-coloring of $G$.
		\end{definition}
		For example, the $vi$-simultaneous coloring of $C_4$ in Figure \ref{C4} is a $vi$-simultaneous $(4,1)$-coloring and so $\chi_{vi,1}(C_4)=4$. Observe that $\chi_{vi,1}(G)\geq\chi_{vi,2}(G)\geq\cdots\geq\chi_{vi,\Delta}(G)=\chi_{vi}(G)$ for every graph $G$ with maximum degree $\Delta$.
		\subsection{Fractional power of graph}
		For the edge coloring and total coloring of any graph $G$, two corresponding graphs are defined. In the line graph of $G$, denoted by $\mathcal{L}(G)$, the vertex set is $E(G)$ and two vertex $e$ and $e'$ are adjacent if $e\cap e'\neq\varnothing$. In the total graph of $G$, denoted by $\mathcal{T}(G)$, vertex set is $V(G)\cup E(G)$ and two vertices are adjacent if and only if they are adjacent or incident in $G$. According to these definitions, we have $\chi'(G)=\chi(\mathcal{L}(G))$ and $\chi''(G)=\chi(\mathcal{T}(G))$. Therefore, edge coloring and total coloring of graphs can be converted to vertex coloring of graphs.\\
		Motivated by the concept of total graph, the fractional power of a graph was first introduced in \cite{paper13}. Let $G$ be a graph and $k$ be a positive integer. The \emph{$k$-power of $G$}, denoted by $G^k$, is defined on the vertex set $V(G)$ by adding edges joining any two distinct vertices $x$ and $y$ with distance at most $k$. Also the $k$-subdivision of $G$, denoted by $G^{\frac{1}{k}}$, is constructed by replacing each edge $xy$ of $G$ with a path of length $k$ with the vertices $x=(xy)_0,(xy)_1,\ldots, (xy)_{k-1},y=(xy)_k$. Note that the vertex $(xy)_l$ has distance $l$ from the vertex $x$, where $l\in \{0,1,\ldots,k\}$. Also, $(xy)_l=(yx)_{k-l}$, for any $l\in \{0,1,\ldots,k\}$. The vertices $(xy)_0$ and $(xy)_k$ are called terminal vertices and the others are called internal vertices. We refer to these vertices in short, $t$-vertices and  $i$-vertices of $G$, respectively. Now the fractional power of graph $G$ is defined as follows.
		\begin{definition}\label{def1}
			Let $G$ be a graph and  $m,n\in \mathbb{N}$. The graph $G^{\frac{m}{n}}$ is defined to be the $m$-power of the $n$-subdivision of $G$. In other words, $G^{\frac{m}{n}}=(G^{\frac{1}{n}})^m$.
		\end{definition} 
		The sets of terminal and internal vertices of $G^\frac{m}{n}$ are denoted by $V_t(G^\frac{m}{n})$ and $V_i(G^\frac{m}{n})$, respectively. It is worth noting that, $G^{\frac{1}{1}}=G$ and $G^{\frac{2}{2}}=\mathcal{T}(G)$.\\
		By virtue of Definition \ref{def1}, one can show that $\omega(G^{\frac{2}{2}})=\Delta(G)+1$ and the Total Coloring Conjecture can be reformulated as follows.
		\begin{conjecture}\label{conj1}
			{For any simple graph $G$, $\chi(G^{\frac{2}{2}})\leq \omega(G^{\frac{2}{2}})+1$.}
		\end{conjecture}
		In \cite{paper13}, the chromatic number of some fractional powers of graphs was first studied and it was proved that $\chi(G^{\frac{m}{n}})=\omega(G^{\frac{m}{n}})$ where $n=m+1$ or $m=2<n$. Also it was conjectured that $\chi(G^{\frac{m}{n}})=\omega(G^{\frac{m}{n}})$ for any graph $G$ with $\Delta(G)\geq3$ when $\frac{m}{n}\in\mathbb{Q}\cap(0,1)$. This conjecture was disproved by Hartke, Liu and Petrickova \cite{hartke2013} who proved that the conjecture is not true for the cartesian product $C_3\Box K_2$ (triangular prism) when $m=3$ and $n=5$. However, they claimed that the conjecture is valid except when $G=C_3\Box K_2$. In addition they proved that the conjecture is true when $m$ is even.\\
		It can be easily seen that, $G$ and $\mathcal{I}(G)$ are isomorphic to the induced subgraphs of $G^\frac{3}{3}$ by $V_t(G^\frac{3}{3})$ and $V_i(G^\frac{3}{3})$, the sets of terminal and internal vertices of $G^\frac{3}{3}$ respectively. So $\chi_i(G)=\chi(G^{\frac{3}{3}}[V_i(G^\frac{3}{3})])$. Also, by considering the $3$-subdivision of a graph $G$, two internal vertices $(uv)_1$ and $(uv)_2$ of the edge $uv$ in $G^{\frac{3}{3}}$ are corresponding to the incidences of the edge $\{u,v\}$ in $G$. For convenience, we denote $(uv)_1$ and $(uv)_2$ with $(u,v)$ and $(v,u)$, respectively.\\
		Similar to the equality $\chi''(G)=\chi(G^{\frac{2}{2}})$, we have the following basic theorem about the relation between  $vi$-simultaneous coloring of a graph and vertex coloring of its $\frac{3}{3}$ power.
		\begin{theorem}\label{vi-simultaneous}
			For any graph $G$, $\chi_{vi}(G)=\chi(G^{\frac{3}{3}})$.
		\end{theorem}
		Because of Theorem~\ref{vi-simultaneous}, we use the terms $\chi_{vi}(G)$ and $\chi(G^{\frac{3}{3}})$ interchangebly in the rest of the paper.  We often use the notation $\chi_{vi}(G)$ to express the theorems and the notation $\chi(G^{\frac{3}{3}})$ in the proofs.\\
		As mentioned in \cite{paper13}, one can easily show that $\omega(G^{\frac{3}{3}})=\Delta(G)+2$, when $\Delta(G)\geq 2$ and $\omega(G^{\frac{3}{3}})=4$, when $\Delta(G)=1$. Therefore, $\Delta+2$ is a lower bound for $\chi(G^{\frac{3}{3}})$ and $\chi_{vi}(G)$, when $\Delta(G)\geq 2$. In \cite{paper13}, the chromatic number of fractional power of cycles and paths are considered, which can be used to show that the graphs with maximum degree two are $vi$-simultaneous 5-colorable (see Section \ref{sec4}). In \cite{iradmusa2020,3power3subdivision} it is shown that $\chi(G^{\frac{3}{3}})\leq7$ for any graph $G$ with maximum degree $3$. Moreover, in \cite{mahsa} it is proved that $\chi(G^{\frac{3}{3}})\leq 9$ for any graph $G$ with maximum degree $4$. Also in \cite{iradmusa2020} it is proved that $\chi(G^{\frac{3}{3}})\leq\chi(G)+\chi_i(G)$ when $\Delta(G)\leq2$ and $\chi(G^{\frac{3}{3}})\leq \chi(G)+\chi_i(G)-1$ when $\Delta(G)\geq 3$. In addition, in \cite{Bruldy}, it is shown that $\chi_i(G)\leq2\Delta(G)$ for any graph $G$. Hence, if $G$ is a graph with $\Delta(G)\geq2$, then $\chi(G^{\frac{3}{3}})=\chi_{vi}(G)\leq 3\Delta(G)$.\\
		According to the results mentioned in the previous paragraph, the following conjecture is true for graphs with maximum degree at most $4$.
		\begin{conjecture}{\em{\cite{mahsa}}}\label{cmahsa}
			Let $G$ be a graph with $\Delta(G)\geq 2$. Then $\chi_{vi}(G)\leq 2\Delta(G)+1$.
		\end{conjecture}
		We know that $\chi(G^{\frac{3}{3}})\geq \omega(G)=\Delta(G)+2$ when $\Delta(G)\geq 2$. In addition, Total Coloring Conjecture states that $\chi(G^{\frac{2}{2}})\leq \Delta(G)+2$. Therefore if Total Coloring Conjecture is correct, then the following conjecture is also true.
		\begin{conjecture}{\em{\cite{mahsa}}}\label{tcmahsa}
			Let $G$ be a graph with $\Delta(G)\geq 2$. Then $\chi(G^{\frac{2}{2}})\leq\chi(G^{\frac{3}{3}})$.
		\end{conjecture}
		Similar to the graphs $\mathcal{L}(G)$, $\mathcal{T}(G)$ and $\mathcal{I}(G)$, for any graph $G$, we can define a corresponding graph, denoted by $\mathcal{T}_{vi,1}(G)$, such that $\chi_{vi,1}(G)=\chi(\mathcal{T}_{vi,1}(G))$.
		\begin{definition}\label{Tvi1}
			Let $G$ be a nonempty graph. The graph $\mathcal{T}_{vi,1}(G)$, is a graph with vertex set $V(G)\times [2]$ and two vertices $(v,i)$ and $(u,j)$ are adjacent in $\mathcal{T}_{vi,1}(G)$ if and only if one of the following conditions hold:
			\begin{itemize}
				\item $i=j=1$ and $d_G(v,u)=1$,
				\item $i=j=2$ and $1\leq d_G(v,u)\leq 2$,
				\item $i\neq j$ and $0\leq d_G(v,u)\leq 1$,
			\end{itemize}
		\end{definition}
		\begin{example}\label{Ex:Tvi1C6}
			{\rm As an example, $\mathcal{T}_{vi,1}(C_6)$ shown in Figure \ref{Tvi1C6}. Unlabeled vertices belong to $V(C_6)\times\{2\}$.
		}\end{example}
		\begin{figure}[h]
			\begin{center}
				\resizebox{7.7cm}{5cm}{%
					\begin{tikzpicture}[scale=0.5]
						\tikzset{vertex/.style = {shape=circle,draw, line width=1pt, opacity=1.0, inner sep=2pt}}
						\tikzset{edge/.style = {-,> = latex', line width=1pt,opacity=1.0}}
						\node [vertex] (0) at (0, 2.5) {};
						\node [vertex] (1) at (3, 2.5) {};
						\node [vertex] (2) at (5, 0) {};
						\node [vertex] (3) at (-2, 0) {};
						\node [vertex] (4) at (3, -2.5) {};
						\node [vertex] (5) at (0, -2.5) {};
						\node [vertex] (6) at (4, 4) {};
						\node  at (5.5,4) {$(v_2,1)$};
						\node [vertex] (7) at (7, 0) {};
						\node  at (8.5,0) {$(v_1,1)$};
						\node [vertex] (8) at (4, -4) {};
						\node  at (5.5,-4) {$(v_6,1)$};
						\node [vertex] (9) at (-1, -4) {};
						\node  at (-2.5,-4) {$(v_5,1)$};
						\node [vertex] (10) at (-4, 0) {};
						\node  at (-5.5,0) {$(v_4,1)$};
						\node [vertex] (11) at (-1, 4) {};
						\node  at (-2.5,4) {$(v_3,1)$};
						\draw [edge] (1) to (2);
						\draw [edge] (1) to (0);
						\draw [edge] (0) to (3);
						\draw [edge] (2) to (4);
						\draw [edge] (4) to (5);
						\draw [edge] (5) to (3);
						\draw [edge] (6) to (11);
						\draw [edge] (11) to (10);
						\draw [edge] (10) to (9);
						\draw [edge] (9) to (8);
						\draw [edge] (8) to (7);
						\draw [edge] (7) to (6);
						\draw [edge] (1) to (6);
						\draw [edge] (2) to (7);
						\draw [edge] (4) to (8);
						\draw [edge] (5) to (9);
						\draw [edge] (3) to (10);
						\draw [edge] (0) to (11);
						\draw [edge] (0) to (6);
						\draw [edge] (11) to (1);
						\draw [edge] (1) to (7);
						\draw [edge] (2) to (6);
						\draw [edge] (2) to (8);
						\draw [edge] (4) to (7);
						\draw [edge] (4) to (9);
						\draw [edge] (5) to (8);
						\draw [edge] (5) to (10);
						\draw [edge] (3) to (9);
						\draw [edge] (10) to (0);
						\draw [edge] (3) to (11);
						\draw [edge] (1) to (4);
						\draw [edge] (2) to (5);
						\draw [edge] (4) to (3);
						\draw [edge] (5) to (0);
						\draw [edge] (3) to (1);
						\draw [edge] (0) to (2);
				\end{tikzpicture}}
				\caption{$\mathcal{T}_{vi,1}(C_6)$}
				\label{Tvi1C6}
			\end{center}
		\end{figure}
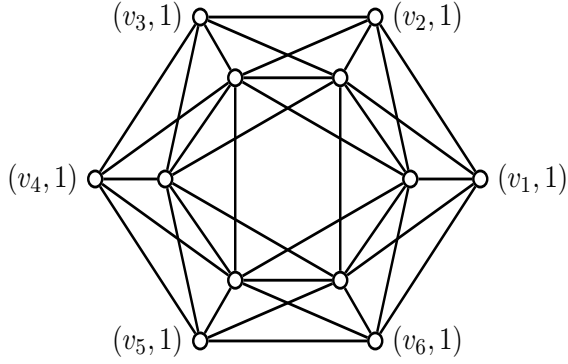
		\begin{theorem}\label{start2}
			For any nonempty graph $G$, $\chi_{vi,1}(G)=\chi(\mathcal{T}_{vi,1}(G))$.
		\end{theorem}
		An incidence coloring of a graph can be viewed as a proper arc coloring of a corresponding digraph. For a graph $G$, digraph $\overrightarrow{G}$ is a digraph obtained from $G$ by replacing each edge of $E(G)$ by two opposite arcs. Any incidence $(v,e)$ of $I(G)$, with $e=\{v,w\}$, can then be associated with the arc $(v,w)$ in $A(\overrightarrow{G})$. Therefore, an incidence coloring of $G$ can be viewed as a proper arc coloring of $\overrightarrow{G}$ satisfying $(i)$ any two arcs having the same tail vertex are assigned distinct colors and $(ii)$ any two consecutive arcs are assigned distinct colors.\\
		Similar to incidence coloring, there is another equivalent coloring for proper coloring of $\frac{3}{3}$-power of a graph or equivalently $vi$-simultaneous proper coloring. 
		\begin{definition}\label{underlying}
			Let $G$ be a graph, $S=S_t\cup S_i$ be a subset of $V(G^{\frac{3}{3}})$ such that $S_t\subseteq V_t(G^{\frac{3}{3}})$ and $S_i\subseteq V_i(G^{\frac{3}{3}})$ and $H$ be the subgraph of $G^{\frac{3}{3}}$ induced by $S$. Also let $A(S_i)=\{(u,v)\ |\ (uv)_1\in S_i\}$ and $V(S_i)=\{u\in V(G)\ |\ I(u)\cap S_i\neq\varnothing\}$. The underlying digraph of $H$, denoted by $D(H)$, is a digraph with vertex set $S_t\cup V(S_i)$ and arc set $A(S_i)$. Specially, $D(G^{\frac{3}{3}})=\overrightarrow{G}$.
		\end{definition}
		Now any proper coloring of $G^{\frac{3}{3}}$ (or, equivalently, any $vi$-simultaneous coloring of $G$) can be viewed as a coloring of vertices and arcs of $D(G^{\frac{3}{3}})$ satisfying $(i)$ any two adjacent vertices are assigned distinct colors, $(ii)$ any arc and its head and tail are assigned distinct colors, $(iii)$ any two arcs having the same tail vertex (of the form $(u,v)$ and $(u,w)$) are assigned distinct colors and $(iv)$ any two consecutive arcs (of the form $(u,v)$ and $(v,w)$) are assigned distinct colors.\\
		A star is a tree with diameter at most two. A star forest is a forest, whose connected components are stars. The star arboricity $st(G)$ of a graph $G$ is the minimum number of star forests in $G$ whose union covers all edges of $G$. In \cite{planarinc} it was proved that  $\chi_i(G)\leq \chi'(G)+st(G)$. Similar to this result, we can give an upper bound for $\chi_{vi}(G)$ in terms of total chromatic number and star arboricity.
		\begin{theorem}\label{start1}
			For any graph $G$, we have $\chi_{vi}(G)\leq \chi(G^{\frac{2}{2}})+st(G)$.
		\end{theorem}
		The aim of this paper is to find exact value or upper bound for the $vi$-simultaneous chromatic number of some classes of graphs by coloring the vertices of $G^{\frac{3}{3}}$ and checking the truthness of the conjecture \ref{cmahsa} for some classes of graphs. We show that the Conjecture~\ref{cmahsa} is true for some graphs such as trees, complete graphs and bipartite graphs. Also we study the relationship between $vi$-simultaneous chromatic number and the other parameters of graphs.
		\subsection{Structure of the paper}
		After this introductory section where we established the background, purpose and some basic definitions and theorems of the paper, we divide the paper into four sections. In Section \ref{sec2}, we prove Theorems \ref{vi-simultaneous}, \ref{start2} and \ref{start1} and some basic lemmas and theorems. In Section \ref{sec3}, we give an upper bound for $vi$-simultaneous chromatic number of a $k$-degenerated graph in terms of $k$ and the maximum degree of graph. In Section \ref{sec4} we provide exact value for chromatic number of $\frac{3}{3}$-powers of cycles, complete graphs and complete bipartite graphs and also give an upper bound for chromatic number of $\frac{3}{3}$-powers of bipartite graphs and conclude that the Conjecture~\ref{cmahsa} is true for these classes of graphs. 
		\section{Basic theorems and lemmas}\label{sec2}
		At first, we prove Theorems \ref{vi-simultaneous}, \ref{start2} and \ref{start1}.\\
		\textbf{Proof of Thorem \ref{vi-simultaneous}} At first, suppose that $\chi(G^{\frac{3}{3}})=k$ and $c:V(G^{\frac{3}{3}})\longrightarrow[k]$ is a proper coloring of $G^{\frac{3}{3}}$. We show that the following $vi$-simultaneous $k$-coloring of $G$ is proper.
		\[c'(x)=\left\{\begin{array}{cc} c(x) & x\in V(G)=V_t(G^{\frac{3}{3}}),\\ c((uv)_1) & x=(u,v)\in I(G). \end{array}\right.\]
		Since $G$ in an induced subgraph of $G^{\frac{3}{3}}$ by the terminal vertices, $c$ is a proper coloring of $G$. So $c'$ assigns different colors to the adjacent vertices of $G$. Now suppose that $(u,v)$ and $(r,s)$ are adjacent vertices in $\mathcal{I}(G)$. There are three cases:\\
		(i) $(r,s)=(v,u)$. Since $(vu)_1$ and $(uv)_1$ are adjacent in $G^{\frac{3}{3}}$, $c'((u,v))=c((uv)_1)\neq c((vu)_1)=c'((r,s))$.\\
		(ii) $r=u$. Since $d_{G^{\frac{1}{3}}}((uv)_1, (us)_1)=2$, $(uv)_1$ and $(us)_1$ are adjacent in $G^{\frac{3}{3}}$. So in this case, $c'((u,v))=c((uv)_1)\neq c((us)_1)=c'((u,s))$.\\
		(iii) $r=v$. Since $d_{G^{\frac{1}{3}}}((uv)_1, (vs)_1)=3$, $(uv)_1$ and $(vs)_1$ are adjacent in $G^{\frac{3}{3}}$. So in this case, $c'((u,v))=c((uv)_1)\neq c((vs)_1)=c'((v,s))$.\\
		Finally suppose that $u\in V(G)$ and $(r,s)\in I(G)$ are incident. So $u=r$ or $u=s$. In the first case, we have $d_{G^{\frac{1}{3}}}(u, (rs)_1)=1$ and in the second case we have $d_{G^{\frac{1}{3}}}(u, (rs)_1)=2$ and $u$ and $(rs)_1$ are adjacent in $G^{\frac{3}{3}}$. So $c'(u)=c(u)\neq c((rs)_1)=c'((r,s))$.\\
		Similarly we can show that each proper $vi$-simultaneous $k$-coloring of $G$ give us a proper $k$-coloring of $G^{\frac{3}{3}}$. Therefore $\chi_{vi}(G)=\chi(G^{\frac{3}{3}})$.
		\hfill $\blacksquare$\\\\
		\textbf{Proof of Thorem \ref{start2}} Firstly, suppose that $\chi_{vi,1}(G)=k$ and $c:V(G)\cup I(G)\longrightarrow [k]$ is a $vi$-simultaneous $(k,1)$-coloring of $G$. We show that the following $k$-coloring of $\mathcal{T}_{vi,1}(G)$ is proper.
		\[c'(x)=\left\{\begin{array}{cc} c(u) & x=(u,1),\\ s & x=(u,2), s\in c(I_2(u)). \end{array}\right.\]
		Since $c$ is a $vi$-simultaneous $(k,1)$-coloring, $|c(I_2(u))|=1$ for any vertex $u\in V(G)$ and so $c'$ is well-defined. Now suppose that $(v,i)$ and $(u,j)$ are adjacent in $\mathcal{T}_{vi,1}(G)$.
		\begin{itemize}
			\item If $i=j=1$, then $c'((v,i))=c(v)\neq c(u)=c'((u,j))$.
			\item If $i=j=2$ and $d_G(v,u)=1$, then $c'((v,i))=c(u,v)\neq c((v,u))=c'((u,j))$.
			\item If $i=j=2$ and $d_G(v,u)=2$, then $c'((v,i))=c(z,v)\neq c((z,u))=c'((u,j))$ where $z\in N_G(v)\cap N_G(u)$.
			\item If $i=1$, $j=2$ and $v=u$, then $c'((v,i))=c(v)\neq c((z,v))=c'((u,j))$ where $z\in N_G(v)$.
			\item If $i=1$, $j=2$ and $d_G(v,u)=1$, then $c'((v,i))=c(v)\neq c((v,u))=c'((u,j))$.
		\end{itemize}
		So $c'$ assigns different colors to the adjacent vertices of $\mathcal{T}_{vi,1}(G)$.\\
		Now suppose that $\chi(\mathcal{T}_{vi,1}(G))=k$ and $c':V(\mathcal{T}_{vi,1}(G))\longrightarrow [k]$ is a proper $k$-coloring of $\mathcal{T}_{vi,1}(G)$. Easily one can show that the following $k$-coloring is a $vi$-simultaneous $(k,1)$-coloring of $G$.
		\[c(x)=\left\{\begin{array}{cc} c'((x,1)) & x\in V(G),\\ c'((v,2)) & x=(u,v)\in I(G). \end{array}\right.\]
		Thus $\chi_{vi,1}(G)=\chi(\mathcal{T}_{vi,1}(G))$.
		\hfill $\blacksquare$\\\\
		\noindent\textbf{Proof of Thorem \ref{start1}} Let $G$ be an undirected graph with star arboricity $st(G)$ and $s \hspace{1mm}:\hspace{1mm} E(G) \longrightarrow [st(G)]$ be a mapping such that $s^{-1}(i)$ is a forest of stars for any $i$, $1\leq i \leq st(G)$. Also, suppose that $c$ be a total coloring of $G^{\frac{2}{2}}$ with colors $\{st(G)+1,\ldots,st(G)+\chi''(G)\}$. Now, to color $t$-vertices and $i$-vertices of the graph $G$, define the mapping $c'$ by $c'((u,v))=s(uv)$ if $v$ is the center of a star in some forest $s^{-1}(i)$. If some star is reduced to one edge, we arbitrarily choose one of its end vertices as the center. Note that, for any edge $uv$, one of the $t$-vertices $u$ or $v$ is the center of a some star forest. It is enough to color the other $t$-vertices and $i$-vertices of $G$.\\
		Consider the graph $G$ on uncolord $t$-vertices and uncolord $i$-vertices. It can be easily seen that the resulting graph, $G'$, is isomorphic to $G^{\frac{2}{2}}$. Now, assign colors $c(u)$ and $c((u,v))$ to a $t$-vertex $u$ and a $i$-vertex $(u,v)$ in $G'$. Therefore,  we have $\chi(G^{\frac{3}{3}})\leq\chi(G^{\frac{2}{2}})+st(G)$.
		\hfill $\blacksquare$\\\\
		For any star forest $F$, we have $st(F)=1$, $\chi(F^{\frac{2}{2}})=\Delta(F)+1$ and $\chi(F^{\frac{3}{3}})=\Delta(F)+2$. Therefore, the upper bound of Theorem \ref{start1} is tight.\\
		The following lemmas will be used in the proofs of some theorems in the next sections. The set $\{c(a)\ |\ a\in A\}$ is denoted by $c(A)$ where $c:D\rightarrow R$ is a function and $A\subseteq D$.
		\begin{lemma}\label{firstlem}
			Let $G$ be a graph with maximum degree $\Delta$ and $c$ is a proper $(\Delta+2)$-coloring of $G^{\frac{3}{3}}$ with colors from $[\Delta+2]$. Then $|c(I_2(v))\leq\Delta-d_G(v)+1$ for any $t$-vertex $v$. Specially $|c(I_2(v))|=1$ for any $\Delta$-vertex $v$ of $G$.
		\end{lemma}
		\begin{proof}{
				Let $v$ be a $t$-vertex of $G$. Since all vertices in $I_1[v]$ are pairwise adjacent in $G^{\frac{3}{3}}$, there are exactly $d_G(v)+1$ colors in $c(I_1[v])$. Now, consider the vertices in $I_2(v)$. Since any vertex in $I_2(v)$ is adjacent with each vertex of $I_1[v]$, the only available colors for these $i$-vertices is the remain colors from $[\Delta+2]\setminus c(I_1[v])$. Therefore, $|c(I_2(v))|\leq\Delta-d_G(v)+1$.
		}\end{proof}
		\begin{lemma}\label{secondlem}
			Let $G$ be a graph, $e$ be a cut edge of $G$ and $C_1$ and $C_2$ be two components of $G-e$. Then $\chi_{vi,l}(G)=\max\{\chi_{vi,l}(H_1),\chi_{vi,l}(H_2)\}$ where $H_i=C_i+e$ for $i\in\{1,2\}$ and $1\leq l\leq\Delta(G)$.
		\end{lemma}
		\begin{proof}{
				Obviously $\chi_{vi,l}(H_1)\leq \chi_{vi,l}(G)$ and $\chi_{vi,l}(H_2)\leq \chi_{vi,l}(G)$. So $\max\{\chi_{vi,l}(H_1),\chi_{vi,l}(H_2)\}\leq\chi_{vi,l}(G)$. Now suppose that $\chi_{vi,l}(H_1)=k_1\geq k_2=\chi_{vi,l}(H_2)$. We show that $\chi_{vi,l}(G)\leq k_1$. Let $c_i:V(H_i)\rightarrow [k_i]$ be a $vi$-simultaneous $(k_i,l)$-colorings ($1\leq i\leq2$) and $e=\{u,v\}$. Since $V(H_1)\cap V(H_2)=\{u, (u,v), (v,u), v\}$ and these four vertices induce a clique, so by suitable permutation on the colors of the coloring $c_1$, we reach to the new coloring $c'_1$ such that $c'_1(x)=c_2(x)$ for any $x\in\{u, (u,v), (v,u), v\}$. Now we can easily prove that the following coloring is a $vi$-simultaneous $(k_1,l)$-coloring:
				\[c(x)=\left\{\begin{array}{cc} c'_1(x) & x\in V(H_1),\\ c_2(x) & x\in V(H_2). \end{array}\right.\]
		}\end{proof}
		\begin{lemma}\label{thirdlem}
			Let $G_1$ and $G_2$ be two graphs, $V(G_1)\cap V(G_2)=\{v\}$ and $G=G_1\cup G_2$. Then
			\[\chi_{vi,1}(G)=\max\{\chi_{vi,1}(G_1),\chi_{vi,1}(G_2), d_G(v)+2\}.\]
		\end{lemma}
		\begin{proof}{
				Suppose that $k=\max\{\chi_{vi,1}(G_1),\chi_{vi,1}(G_2), d_G(v)+2\}$. Obviously $\chi_{vi,1}(G_1)\leq \chi_{vi,1}(G)$, $\chi_{vi,1}(G_2)\leq \chi_{vi,1}(G)$ and $d_G(v)+2\leq\Delta(G)+2\leq\chi_{vi}(G)\leq\chi_{vi,1}(G)$. So $k\leq\chi_{vi,1}(G)$. Now suppose that $c_1$ and $c_2$ are $vi$-simultaneous $(k,1)$-coloring of $G_1$ and $G_2$ respectively. Note that $I_1^{G_1}[v]$, $I_1^{G_2}[v]$ and $I_1^{G}[v]$ are cliques and $I_2^{G_1}(v)$, $I_2^{G_2}(v)$ and $I_2^{G}(v)$ are independent sets in $G_1$, $G_2$ and $G$ respectively. Also $c_i(I_1^{G_i}[v])\cap c_i(I_2^{G_i}(v))=\varnothing$ and $|c_i(I_2^{G_i}(v))|=1$ for each $i\in [2]$. So by suitable permutations on the colors of $c_2$ in three steps, we reach to the new coloring $c_3$:
				\begin{itemize}
					\item [(1)] If $c_1(v)=a\neq b=c_2(v)$ then we just replace colors $a$ and $b$ together in $c_2$ and otherwise we do nothing. We denote the new coloring by $c'_2$.
					\item [(2)] Let $c_1(x)=c$ and $c'_2(y)=d$ for each $x\in I_2^{G_1}(v)$ and $y\in I_2^{G_2}(v)$. If $c\neq d$ then we just replace colors $c$ and $d$ together in $c'_2$. Otherwise we do nothing. We denote the new coloring by $c''_2$. Obviously, $c\neq a\neq d$ and so $c''_2(v)=a$.
					\item [(3)] If $c''_2(I_1^{G_2}(v))\cap c_1(I_1^{G_1}(v))=\varnothing$ we do nothing. Otherwise, suppose that $c''_2(I_1^{G_2}(v))\cap c_1(I_1^{G_1}(v))=\{a_1,\ldots,a_s\}$. Since $k\geq d_G(v)+2$ and $|c''_2(I_{G_2}[v])\cup c_1(I_{G_1}[v])|=d_{G}(v)+2-s$, there are $s$ colors $b_1,\ldots,b_s$ which have not appeared in $c''_2(I_{G_2}[v])\cup c_1(I_{G_1}[v])$. Now we replace $a_i$ and $b_i$ together for each $i\in\{1,\ldots,s\}$. We denote the new coloring by $c_3$.
				\end{itemize}
				Now we can easily show that the following function is a $vi$-simultaneous proper $(k,1)$-coloring for $G$:
				\[c(x)=\left\{\begin{array}{cc} c_1(x) & x\in V(G_1)\cup I(G_1),\\ c_3(x) & x\in V(G_2)\cup I(G_2). \end{array}\right.\]
		}\end{proof}
		\begin{theorem}\label{blocks}
			Let $k\in\mathbb{N}$ and $G$ be a graph with blocks $B_1,\ldots,B_k$. Then
			\[\chi_{vi,1}(G)=\max\{\chi_{vi,1}(B_1),\ldots,\chi_{vi,1}(B_k), \Delta(G)+2\}.\]
			Specially, $\chi_{vi,1}(G)=\max\{\chi_{vi,1}(B_1),\ldots,\chi_{vi,1}(B_k)\}$ when $G$ has at least one $\Delta(G)$-vertex which is not cut vertex.
		\end{theorem}
		\begin{proof}{
				By induction on the number $k$ and applying Lemma \ref{thirdlem}, the proof will be done. 
		}\end{proof}
		We can determine an upper bound on the $vi$-simultaneous chromatic number $\chi_{vi,s}(G)$ in terms of $\Delta(G)$ and list chromatic number of $G$.\\
		\begin{definition}\label{listcoloring}\cite{bondy}
			Let $G$ be a graph and $L$ be a function which assigns to each vertex $v$ of $G$ a set $L(v)\subset\mathbb{N}$, called the list of $v$. A coloring $c:V(G)\rightarrow\mathbb{N}$ such that $c(v)\in L(v)$ for all $v\in V(G)$ is called a list coloring of $G$ with respect to $L$, or an $L$-coloring, and we say that $G$ is $L$-colorable. A graph $G$ is $k$-list-colorable if it has a list coloring whenever all the lists have length $k$. The smallest value of $k$ for which $G$ is $k$-list-colorable is called the list chromatic number of $G$, denoted $\chi_{l}(G)$.
		\end{definition}
		\begin{theorem}\label{upperbound-list}
			Let $G$ be a  nonempty graph and $s\in\mathbb{N}$. Then\\
			(i) $\chi_{vi,s}(G)\leq\max\{\chi_{i,s}(G),\chi_{l}(G)+\Delta(G)+s\}$,\\
			(ii) If $\chi_{i,s}(G)\geq\chi_{l}(G)+\Delta(G)+s$, then $\chi_{vi,s}(G)=\chi_{i,s}(G)$.
		\end{theorem}
		\begin{proof}{
				(i) Suppose that $\max\{\chi_{i,s}(G),\chi_{l}(G)+\Delta(G)+s\}=k$. So there exists an incidence $(k,s)$-coloring $c_i: I(G)\rightarrow [k]$ of $G$ and
				hence $|c_i(I_2(u))|\leq s$ for any vertex $u\in V(G)$. Therefore, $|c_i(I_G(u))|\leq \Delta(G)+s$. Now we extend $c_i$ to a $vi$-simultaneous $(k,s)$-coloring $c$ of $G$. The set of available colors for the vetex $u$ is $L(u)=[k]\setminus c_i(I_G(u))$ which has at least $k-\Delta(G)-s\geq \chi_l(G)$ colors. Since $|L(u)|\geq\chi_{l}(G)$ for any vertex $u\in V(G)$, there exists a proper vertex coloring $c_v$ of $G$ such that $c_v(u)\in L(u)$. Now one can easily show that the following coloring is a $vi$-simultaneous $(k,s)$-coloring of $G$:
				\[c(x)=\left\{\begin{array}{cc} c_i(x) & x\in I(G),\\ c_v(x) & x\in V(G). \end{array}\right.\]
				(ii) If $\chi_{i,s}(G)\geq\chi_{l}(G)+\Delta(G)+s$, then $\chi_{vi,s}(G)\leq\chi_{i,s}(G)$. In addition, any $vi$-simultaneous $(k,s)$-coloring of $G$ induces an incidence $(k,s)$-coloring of $G$ and so $\chi_{i,s}(G)\leq\chi_{vi,s}(G)$. Therefore, $\chi_{vi,s}(G)=\chi_{i,s}(G)$.
		}\end{proof}
		\begin{corollary}\label{upperbound-list-vi1}
			$\chi_{vi,1}(G)\leq\max\{\chi(G^2),\chi_{l}(G)+\Delta(G)+1\}$ for any nonempty graph $G$. Specially, if $\chi(G^2)\geq\chi_{l}(G)+\Delta(G)+1$, then $\chi_{vi,1}(G)=\chi(G^2)$.
		\end{corollary}
		\begin{corollary}\label{upperbound-diam-vi1}
			Let $G$ be a graph of order $n$ with $diam(G)=2$. Then $\chi_{vi,1}(G)\leq\max\{n, \chi_l(G)+\Delta(G)+1\}$. Specially if $\Delta(G)\leq\frac{n}{2}-1$, then $\chi_{vi,1}(G)=n$.
		\end{corollary}
		\begin{remark}{\rm
In \cite{Cranston}, it was proved that the square of any cubic graph other than the Petersen graph is 8-list-colorable and so $\chi(G^2)\leq8$. In addition the diameter of the Petersen graph $P$ is two. Therefore, by Corollaries \ref{upperbound-list-vi1} and \ref{upperbound-diam-vi1}, $\chi_{vi,1}(P)=10$ for the Petersen graph and $\chi_{vi,1}(G)\leq 8$ for any graph $G$ with $\Delta(G)=3$ other than the Petersen graph.
		}\end{remark}
		\section{$k$-degenerated graphs}\label{sec3}
		A graph $G$ is said to be $k$-degenerated if any subgraph of $G$ contains a vertex of degree at most $k$. For example, Any graph $G$ is 1-degenerated if and only if $G$ is a forest. We can give an upper bound for $vi$-simultaneous chromatic number of a $k$-degenerated graph in terms of $k$ and its maximum degree.\\
		Let $\mathcal{F}=\{A_1,\ldots,A_n\}$ be a finite family of $n$ subsets of a finite set $X$. A system of distinct representatives (SDR) for the family $\mathcal{F}$ is a set $\{a_1,\ldots,a_n\}$ of distinct elements of $X$ such that $a_i\in A_i$ for all $i\in [n]$.
		\begin{theorem}\label{kdegenerated}
			Let $k\in\mathbb{N}$ and $G$ be a $k$-degenerated graph with $\Delta(G)\geq2$. Then $\chi_{vi,k}(G)\leq \Delta(G)+2k$. 
		\end{theorem}
		\begin{proof}{
				If $k=\Delta(G)$, then $\chi_{vi,k}(G)=\chi_{vi}(G)\leq 3\Delta(G)=\Delta(G)+2k$. So we suppose that $1\leq k\leq\Delta(G)-1$. Assume the contrary, and let the theorem is false and $G$ be a minimal counter-example. Let $u$ be a vertex in $G$ with degree $r\leq k$ and $N_G(u)=\{u_1,\ldots,u_r\}$ and let $G'=G-u$. According to the minimality of $G$, $\chi_{vi,k}(G')\leq \Delta(G)+2k$ and there exists a $vi$-simultaneous $(\Delta(G)+2k,k)$-coloring $c'$ of $G'$. We extend $c'$ to a $vi$-simultaneous $(\Delta(G)+2k,k)$-coloring $c$ of $G$ which is a contradiction.\\
				Firstly, we color the vertices of $I_1(u)$. For each $(u,u_i)\in I_1(u)$ there are at least $k$ available colors if $|c'(I_2(u_i))|=k$ and there are at least $2k$ available colors if $|c'(I_2(u_i))|\leq k$. Let $A_i$ be the set of available colors for $(u,u_i)\in I_1(u)$. Since we must select distinct colors for the vertices of $I_1(u)$, we prove that the family $\mathcal{F}=\{A_1,\ldots,A_r\}$ has a system of distinct representatives. Because $|\cup_{j\in J}A_j|\geq k\geq |J|$ for any subset $J\subseteq [r]$, using Hall's Theorem (see Theorem 16.4 in \cite{bondy}), we conclude that $\mathcal{F}$ has an SDR $\{a_1,\ldots,a_r\}$ such that $|\{a_j\}\cup c'(I_2(u_j))|\leq k$ for any $j\in [r]$. We color the vertex $(u,u_j)$ by $a_j$ for any $j\in [r]$. Now we color the vertices of $I_2(u)$. Since $|c'(I_{G'}[u_j]\cup c(I_1^{G}(u))|<\Delta(G)+2k$ for each $j\in [r]$, there exists at least one available color for the vertex $(u_j,u)$. Finally, we select the color of the vertex $u$. Since $|I_G(u)\cup N_G(u)|=3r<\Delta(G)+2k$, we can color the vertex $u$ and complete the coloring of $c$.
		}\end{proof}
		\begin{corollary}\label{tree}
			Let $F$ be a forest. Then
			\[\chi_{vi,1}(F)=\left\{\begin{array}{lll} 1 & \Delta(F)=0,\\ 4 & \Delta(F)=1,\\ \Delta(F)+2 & \Delta(F)\geq2. \end{array}\right.\]
		\end{corollary}
		\begin{proof}{ The proof is trivial for $\Delta(F)\leq1$. So we suppose that $\Delta(F)\geq2$. Each forest is a 1-degenerated graph. So by use of Theorem \ref{kdegenerated} we have $\chi_{vi,1}(F)\leq\Delta(F)+2$. In addition, $\chi_{vi,1}(F)\geq\chi_{vi}(F)=\chi(F^{\frac{3}{3}})\geq\omega(F^{\frac{3}{3}})=\Delta(F)+2$. Hence $\chi_{vi,1}(F)=\Delta(F)+2$.
		}\end{proof}
		\begin{corollary}
			For any $n\in\mathbb{N}\setminus\{1\}$, $\chi_{vi,1}(P_n)=4$.
		\end{corollary}
		\begin{remark}{\rm
				Using the following simple algorithm, we have a proper $(\Delta+2)$-coloring for $\frac{3}{3}$-power of any tree $T$ with $\Delta(T)=\Delta$:\\
				Suppose that $v_1,\ldots,v_n$ are  $t$-vertices of $T$ and the $t$-vertex $v_1$ of degree $\Delta$ is the root of $T$. To achieve a $(\Delta+2)$-coloring of $T^{\frac{3}{3}}$, assign color $1$ to the $v_1$ and color all $i$-vertices in $I_1(v_1)$ with distinct colors in $\{2,\ldots,\Delta+1\}$. Note that, since these $i$-vertices are pairwise adjacent, they must have different colors. Also, color all  $i$-vertices in $I_2(v_1)$ with color $\Delta+2$.\\
				Now, to color the other $t$-vertices and $i$-vertices of $T$, for the $t$-vertex $v_i$ with colored parent $p_{v_i}$, $2\leq i\leq n$, color all the uncolored  $i$-vertices in $I_2(v_i)$ same as $(p_{v_i}v_i)_1$. Then color $v_i$ with a color from $[\Delta+2]\setminus\{c(p_{v_i}),c((p_{v_i}v_i)_1), c((p_{v_i}v_i)_2)\}$. Now, color all the uncolored $i$-vertices in $I_1(v_i)$ with distinct $\Delta-1$ colors from $[\Delta+2]\setminus\{c((p_{v_i}v_i)_1), c((p_{v_i}v_i)_2), c(v_i)\}$.}
		\end{remark}
		As each outerplanar graph is a $2$-degenerated graph and each planar graph is a $5$-degenerated graph, we can result the following corollary by use of the Theorem \ref{kdegenerated}.
\begin{corollary}
		Let $G$ be a graph with maximum degree $\Delta$.
	\begin{itemize}
	\item[(i)] If $G$ is an outerplanar graph, then $\chi_{vi,2}(G)\leq \Delta+4$.
		\item[(ii)] If $G$ is a planar graph, then $\chi_{vi,5}(G)\leq \Delta+10$.
	\end{itemize}
\end{corollary}
		We decrease the upper bound of Theorem \ref{kdegenerated} to $\Delta+5$ for 3-degenerated graphs with maximum degree at least five.
		\begin{theorem}\label{3degenerated}
			Every $3$-degenerated graph $G$ with $\Delta(G)\geq5$ admits a $vi$-simultaneous $(\Delta(G)+5,3)$-coloring. Therefore,
			$\chi_{vi,3}(G)\leq\Delta(G)+5$.
		\end{theorem}
		\begin{proof}{
				Assume the contrary, and let the theorem is false and $G$ be a minimal counter-example. Let $u$ be a vertex in $G$ with degree $r\leq 3$ and $N_G(u)=\{u_1,\ldots,u_r\}$ and let $G'=G-u$. If $\Delta(G')=4$, then by Theorem \ref{kdegenerated} we have $\chi_{vi,3}(G')\leq 4+6=10=\Delta(G)+5$ and if $\Delta(G')\geq 5$, according to the minimality of $G$, $\chi_{vi,3}(G')\leq \Delta(G)+5$. So there exists a $vi$-simultaneous $(\Delta(G)+5,3)$-coloring $c'$ of $G'$. We extend $c'$ to a $vi$-simultaneous $(\Delta(G)+5,3)$-coloring $c$ of $G$, which is a contradiction.\\
				Firstly, we color the vertices of $I_1(u)$. For each $(u,u_i)\in I_1(u)$ there are at least $3$ available colors if $|c'(I_2(u_i))|=3$ and there are at least $5$ available colors if $|c'(I_2(u_i))|\leq 2$. Let $A_i$ be the set of available colors for $(u,u_i)\in I_1(u)$ and $C_i=c'(I_2(u_i))$. Since we must select distinct colors for the vertices of $I_1(u)$, we prove that the family $\mathcal{F}=\{A_1,\ldots,A_r\}$ has an SDR. According to the degree of $u$ and the sizes of $C_1$, $C_2$ and $C_3$, we consider five cases:
				\begin{itemize}
					\item [(1)] $r\leq2$. Since $|A_i|\geq3$, easily one can show that $\mathcal{F}$ has an SDR $\{a_j|\ j\in [r]\}$ such that $|\{a_j\}\cup c'(I_2(u_j))|\leq 3$ for any $j\in [r]$. We color the vertex $(u,u_j)$ by $a_j$ for any $j\in [r]$. Now we color the vertices of $I_2(u)$. Since $|c'(I_{G'}[u_j]\cup c(I_1^{G}(u))|<\Delta(G)+2+r\leq \Delta(G)+4$ for each $j\in [r]$, there exists at least one available color for the vertex $(u_j,u)$. Finally, we select the color of the vertex $u$. Since $|I_G(u)\cup N_G(u)|=3r\leq 6<\Delta(G)+5$, we can color the vertex $u$ and complete the coloring of $c$.
					\item [(2)] $r=3$ and $|C_j|\leq2$ for any $j\in [3]$. Because $|\cup_{j\in J}A_j|\geq 5\geq |J|$ for any subset $J\subseteq [r]$, using Hall's Theorem (see Theorem 16.4 in \cite{bondy}), we conclude that $\mathcal{F}$ has an SDR $\{a_1,\ldots,a_r\}$ such that $|\{a_j\}\cup c'(I_2(u_j))|\leq 3$ for any $j\in [r]$. We color the vertex $(u,u_j)$ by $a_j$ for any $j\in [r]$. Now we color the vertices of $I_2(u)$. Since $|c'(I_{G'}[u_j]\cup c(I_1^{G}(u))|<\Delta(G)+2+r-1\leq \Delta(G)+4$ for each $j\in [r]$, there exists at least one available color for the vertex $(u_j,u)$. Finally, we select the color of the vertex $u$. Since $|I_G(u)\cup N_G(u)|=9<\Delta(G)+5$, we can color the vertex $u$ and complete the coloring of $c$. 	
					\item [(3)] $r=3$ and $|C_j|\leq2$ for two sets of  $C_j$s. Without loss of generality, let $|C_1|=|C_2|=2$ and $|C_3|=3$. If $C_j\cap c'(I_{G'}[u_3])$ is nonempty for some $j\in\{1,2\}$ and $a\in C_j\cap c'(I_{G'}[u_3])$, then we color the vertex $(u,u_j)$ with $a$, the vertex $(u,u_i)$ ($j\neq i\in [2]$) with color $b$ from $C_i\setminus\{a\}$ ($b\in A_i\setminus\{a\}$ if $C_i=\{a\}$) and the vertex $(u,u_3)$ with color $d$ from $C_3\setminus\{a,b\}$.\\
					Because $|c'(I_{G'}[u_3])|=\Delta(G)+3$, if $C_1\cap c'(I_{G'}[u_3])=\varnothing=C_2\cap c'(I_{G'}[u_3])$ then $C_1=C_2$. Suppose that $C_1=C_2=\{a,b\}$ and $d\in A_1\setminus\{a,b\}$ (note that $|A_1|=5$). So $d\in c'(I_{G'}[u_3])$. We color the vertex $(u,u_1)$ with $d$, the vertex $(u,u_2)$ with color $a$ and the vertex $(u,u_3)$ with color $f$ from $C_3\setminus\{a,d\}$. Now we color the vertices of $I_2(u)$. Since $|c'(I_{G'}[u_j]\cup c(I_1^{G}(u))|\leq\Delta(G)+4$ for each $j\in [r]$, there exists at least one available color for the vertex $(u_j,u)$. Finally, we select the color of the vertex $u$. Since $|I_G(u)\cup N_G(u)|=9<\Delta(G)+5$, we can color the vertex $u$ and complete the coloring of $c$.
					\item [(4)]  $r=3$ and $|C_j|\leq2$ for only one set of $C_j$s. Without loss of generality, let $|C_1|=2$ and $|C_2|=|C_3|=3$.
					If $C_1\cap c'(I_{G'}[u_j])$ is nonempty for some $j\in\{2,3\}$ and $a\in C_1\cap c'(I_{G'}[u_j])$, then we color the vertex $(u,u_1)$ with $a$. Suppose that $j\neq i\in\{2,3\}$. Since $|C_i|+|c'(I_{G'}[u_j])|=\Delta(G)+6$, $C_i\cap c'(I_{G'}[u_j])\neq\varnothing$. Let $b\in C_i\cap c'(I_{G'}[u_j])$ and color the vertex $(u,u_i)$ with color $b$ and the vertex $(u,u_j)$ with color $d$ from $C_j\setminus\{a,b\}$.\\
					Because $|c'(I_{G'}[u_2])|=|c'(I_{G'}[u_3])|=\Delta(G)+3$, if $C_1\cap c'(I_{G'}[u_2])=\varnothing=C_1\cap c'(I_{G'}[u_3])$ then $c'(I_{G'}[u_2])=c'(I_{G'}[u_3])$. Since $|C_i|+|c'(I_{G'}[u_j])|=\Delta(G)+6$, $C_i\cap c'(I_{G'}[u_j])\neq\varnothing$ when $\{i,j\}=\{2,3\}$. Therefore, there exist $b\in C_2\cap c'(I_{G'}[u_3])$ and $d\in C_3\cap c'(I_{G'}[u_2])$ such that $b\neq d$. Now we color the vertex $(u,u_1)$ with $a\in C_1$, the vertex $(u,u_2)$ with color $b$ and the vertex $(u,u_3)$ with color $d$. Now we color the vertices of $I_2(u)$. Since $|c'(I_{G'}[u_j]\cup c(I_1^{G}(u))|\leq\Delta(G)+4$ for each $j\in [r]$, there exists at least one available color for the vertex $(u_j,u)$. Finally, we select the color of the vertex $u$. Since $|I_G(u)\cup N_G(u)|=9<\Delta(G)+5$, we can color the vertex $u$ and complete the coloring of $c$. 
					\item [(5)] $r=3$ and $|C_j|=3$ for any $j\in [3]$. For any $i,j\in [3]$, since $|C_i|+|c'(I_{G'}[u_j])|=\Delta(G)+6$, $C_i\cap c'(I_{G'}[u_j])\neq\varnothing$. So there exist $a_1\in C_1\cap c'(I_{G'}[u_2])$, $a_2\in C_2\cap c'(I_{G'}[u_3])$ and $a_3\in C_3\cap c'(I_{G'}[u_1])$.
					If $|\{a_1,a_2,a_3\}|=3$, then we color the vertex $(u,u_j)$ with color $a_j$ ($j\in [3]$) and similar to the previous cases, we can complete the coloring $c$.
					Now suppose that $|\{a_1,a_2,a_3\}|=2$. Without loss of generality, suppose that $a_1=a_2\neq a_3$ and $b\in C_2\setminus\{a\}$. In this case, we color $(u,u_1)$ with $a_1$, the vertex $(u,u_2)$ with color $b$ and the vertex $(u,u_3)$ with color $a_3$.
					Finally suppose that $a_1=a_2=a_3$. If $(C_i\setminus\{a_1\})\cap c'(I_{G'}[u_j])\neq\varnothing$ for some $i,j\in [3]$ and $b\in (C_i\setminus\{a_1\})\cap c'(I_{G'}[u_j])$, we color $(u,u_i)$ with $b$, the vertex $(u,u_2)$ with color $a_1$ and the vertex $(u,u_s)$ with color $d\in C_s\setminus\{a_1,b\}$ where $i\neq s\neq j$. Otherwise, we have $(C_1\setminus\{a_1\})\cap c'(I_{G'}[u_3])=\varnothing=(C_2\setminus\{a_1\})\cap c'(I_{G'}[u_3])$ which concludes $C_1=C_2$. Suppose that $C_1=C_2=\{a_1,b,d\}$. Now we color $(u,u_1)$ with $b$, the vertex $(u,u_2)$ with color $a_1$ and the vertex $(u,u_3)$ with color $f\in C_3\setminus\{a_1,b\}$.\\
					In all of these 3 subcases, we have $|c'(I_{G'}[u_j]\cup c(I_1^{G}(u))|\leq\Delta(G)+4$ for each $j\in [3]$ and similar to the previous cases, we can complete the coloring $c$.  	
				\end{itemize}
		}\end{proof}
		\begin{problem}{\rm
				Let $G$ be a $3$-degenerated graph with $\Delta(G)=4$. We know that $\chi_{vi}(G)\leq9$. What is the sharp upper bound for $\chi_{vi,1}(G)$, $\chi_{vi,2}(G)$ and $\chi_{vi,3}(G)$? By Theorem \ref{kdegenerated},   $\chi_{vi,3}(G)\leq10$. Is this upper bound sharp or similar to Theorem \ref{3degenerated}, the upper bound is 9?
		}\end{problem}
		\section{Cycles, Complete and Bipartite Graphs}\label{sec4}
		In \cite{paper13}, it was proved that $\chi(C_k^m)=k$, when $m\geq \lfloor\frac{k}{2}\rfloor$ and otherwise, $\chi(C_k^m)=\lceil\frac{k}{\lfloor\frac{k}{m+1}\rfloor}\rceil$. With a simple review, we can prove that $\chi(G^{\frac{3}{3}})=\chi_{vi}(G)\leq 5$ when $\Delta(G)=2$ and in this case, $\chi(G^{\frac{3}{3}})=\chi_{vi}(G)=4$ if and only if any component of $G$ is a cycle of order divisible by 4 or a path. In the first theorem, we show that any cycle of order at least four is $vi$-simultaneous $(5,1)$-colorable. To avoid drawing too many edges in the figures, we use $\frac{1}{3}$-powers of graphs instead of $\frac{3}{3}$-powers of graphs. Internal vertices are shown with white color and terminal vertices are shown with color black. 
		\begin{theorem}\label{cycles}
			Let $3\leq n\in\mathbb{N}$. Then
			\[\chi_{vi,1}(C_n)=\left\{\begin{array}{lll} 6 & n=3,\\ 4 & n\equiv 0\ (mod\ 4),\\ 5 & otherwise. \end{array}\right.\]
		\end{theorem}
		\begin{figure}[h]
			\begin{center}
				\begin{tikzpicture}[scale=1.0]
					\tikzset{vertex/.style = {shape=circle,draw, line width=1pt, opacity=1.0, inner sep=2pt}}
					\tikzset{vertex1/.style = {shape=circle,draw, fill=black, line width=1pt,opacity=1.0, inner sep=2pt}}
					\tikzset{arc/.style = {->,> = latex', line width=1pt,opacity=1.0}}
					\tikzset{edge/.style = {-,> = latex', line width=1pt,opacity=1.0}}
					\node[vertex1] (a) at (0,0) {};
					\node  at (0,-0.4) {$1$};
					\node[vertex] (b) at (1,0) {};
					\node  at (1,-0.4) {$2$};
					\node[vertex] (c) at  (2,0) {};
					\node  at (2,-0.4) {$3$};
					\node[vertex1] (d) at  (3,0) {};
					\node  at (3,-0.4) {$4$};
					\node[vertex] (e) at  (2.5,0.85) {};
					\node  at (3,0.85) {$5$};
					\node[vertex] (f) at  (2,1.7) {};
					\node  at (2.5,1.7) {$2$};
					\node[vertex1] (g) at (1.5,2.55) {};
					\node  at (1.9,2.55) {$6$};
					\node[vertex] (h) at (1,1.7) {};
					\node  at (0.6,1.7) {$3$};
					\node[vertex] (i) at (0.5,0.85) {};
					\node  at (0.1,0.85) {$5$};
					\draw[edge] (a)  to (b);
					\draw[edge] (b)  to (c);
					\draw[edge] (c)  to (d);
					\draw[edge] (d)  to (e);
					\draw[edge] (e)  to (f);
					\draw[edge] (f)  to (g);
					\draw[edge] (g)  to (h);
					\draw[edge] (h)  to (i);
					\draw[edge] (i)  to (a);
				\end{tikzpicture}
				\caption{$vi$-simultaneous proper $(6,1)$-coloring of $C_3$. Black vertices are corresponding to the vertices of $G$ and white vertices are corresponding to the incidences of $C_3$.}
				\label{C3}
			\end{center}
		\end{figure}
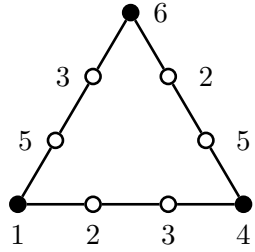
		\begin{proof}{
				Suppose that $V(C_n)=\{v_1,v_2,\ldots,v_n\}$ and $c$ is a $vi$-simultaneous $(k,1)$-coloring of $C_3$. We have $c(v_i)\neq c((v_i,v_j))=c((v_l,v_j))$ where $\{i,j,l\}=[3]$. So
				\[|\{c(v_1),c(v_2),c(v_3), c((v_1,v_2)),c((v_2,v_1)),c((v_1,v_3))\}|=6.\]
				Therefore, $k\geq6$. Figure \ref{C3} shows a $vi$-simultaneous $(6,1)$-coloring of $C_3$ and so $\chi_{vi,1}(C_3)=6$. In the second part, $\chi_{vi}(C_n)=\chi(C_n^{\frac{3}{3}})=\chi(C_{3n}^3)=\lceil\frac{3n}{\lfloor\frac{3n}{4}\rfloor}\rceil=4=\Delta(C_n)+2$ and hence Lemma \ref{firstlem} shows that any $vi$-simultaneous $4$-coloring of $C_n$ is a $vi$-simultaneous $(4,1)$-coloring.\\
				For the last part, we consider three cases:\\
				(i) $n=4q+1$, $q\in\mathbb{N}$. Suppose that $c$ is a $vi$-simultaneous $(4,1)$-coloring of $C_{n-1}$ and
				\[(c(v_1),c((v_1,v_{n-1})), c((v_{n-1},v_1)), c(v_{n-1}))=(1,4,3,2).\]
				In this coloring, the colors of the other vertices uniquely determined. To find a $vi$-simultaneous $(5,1)$-coloring of $C_{n}$, we replace the edge $\{v_1,v_{n-1}\}$ with the path $P=v_{n-1}v_{n}v_1$. Now we define the coloring $c'$ as follows (See Figure \ref{4q+1}):
				\[c'(x)=\left\{\begin{array}{lllll} 2 & x=v_n,\\ 3 & x\in \{v_{n-1}, (v_n,v_1)\},\\ 4 & x=(v_n,v_{n-1}),\\ 5 & x\in\{v_{n-2},(v_1,v_n), (v_{n-1},v_n\},\\ c(x) & otherwise. \end{array}\right.\]
				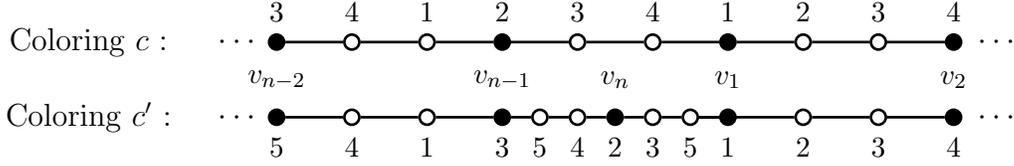
\begin{figure}[h]
					\begin{center}
						\begin{tikzpicture}[scale=1.0]
							\tikzset{vertex/.style = {shape=circle,draw, line width=1pt, opacity=1.0, inner sep=2pt}}
							\tikzset{vertex1/.style = {shape=circle,draw, fill=black, line width=1pt,opacity=1.0, inner sep=2pt}}
							\tikzset{edge/.style = {-,> = latex', line width=1pt,opacity=1.0}}
							\node[vertex1] (a) at (0,0) {};
							\node  at (0,0.4) {$3$};
							\node  at (0,-0.5) {$v_{n-2}$};
							\node[vertex] (b) at (1,0) {};
							\node  at (1,0.4) {$4$};
							\node[vertex] (c) at  (2,0) {};
							\node  at (2,0.4) {$1$};
							\node[vertex1] (d) at  (3,0) {};
							\node  at (3,0.4) {$2$};
							\node  at (3,-0.5) {$v_{n-1}$};
							\node[vertex] (e) at  (4,0) {};
							\node  at (4, 0.4) {$3$};
							\node[vertex] (f) at  (5,0) {};
							\node  at (5,0.4) {$4$};
							\node[vertex1] (g) at (6,0) {};
							\node  at (6,0.4) {$1$};
							\node  at (6,-0.5) {$v_{1}$};
							\node[vertex] (h) at (7,0) {};
							\node  at (7,0.4) {$2$};
							\node[vertex] (i) at (8,0) {};
							\node  at (8,0.4) {$3$};
							\node[vertex1] (j) at (9,0) {};
							\node  at (9,0.4) {$4$};
							\node  at (9,-0.5) {$v_{2}$};
							\node  at (4.5,-0.5) {$v_{n}$};
							\node  at (-0.5,0) {{\large $\cdots$}};
							\node  at (-2.5,0) {{\large Coloring $c$ :}};
							\node  at (9.6,0) {{\large $\cdots$}};
							\node  at (-2.5,-1) {{\large Coloring $c'$ :}};
							\draw[edge] (a)  to (b);
							\draw[edge] (b)  to (c);
							\draw[edge] (c)  to (d);
							\draw[edge] (d)  to (e);
							\draw[edge] (e)  to (f);
							\draw[edge] (f)  to (g);
							\draw[edge] (g)  to (h);
							\draw[edge] (h)  to (i);
							\draw[edge] (i)  to (j);
							\node[vertex1] (a1) at (0,-1) {};
							\node  at (0,-1.4) {$5$};
							\node[vertex] (b1) at (1,-1) {};
							\node  at (1,-1.4) {$4$};
							\node[vertex] (c1) at  (2,-1) {};
							\node  at (2,-1.4) {$1$};
							\node[vertex1] (d1) at  (3,-1) {};
							\node  at (3,-1.4) {$3$};
							\node[vertex] (e1) at  (3.5,-1) {};
							\node  at (3.5, -1.4) {$5$};
							\node[vertex] (f1) at  (4,-1) {};
							\node  at (4,-1.4) {$4$};
							\node[vertex1] (g1) at (4.5,-1) {};
							\node  at (4.5,-1.4) {$2$};
							\node[vertex] (h1) at (5,-1) {};
							\node  at (5,-1.4) {$3$};
							\node[vertex] (i1) at (5.5,-1) {};
							\node  at (5.5,-1.4) {$5$};
							\node[vertex1] (j1) at (6,-1) {};
							\node  at (6,-1.4) {$1$};
							\node[vertex] (k1) at (7,-1) {};
							\node  at (7,-1.4) {$2$};
							\node[vertex] (l1) at (8,-1) {};
							\node  at (8,-1.4) {$3$};
							\node[vertex1] (m1) at (9,-1) {};
							\node  at (9,-1.4) {$4$};
							\node  at (-0.5,-1) {{\large $\cdots$}};
							\node  at (9.6,-1) {{\large $\cdots$}};
							\draw[edge] (a1)  to (b1);
							\draw[edge] (b1)  to (c1);
							\draw[edge] (c1)  to (d1);
							\draw[edge] (d1)  to (e1);
							\draw[edge] (e1)  to (f1);
							\draw[edge] (f1)  to (g1);
							\draw[edge] (g1)  to (h1);
							\draw[edge] (h1)  to (i1);
							\draw[edge] (i1)  to (j1);
							\draw[edge] (i1)  to (k1);
							\draw[edge] (k1)  to (l1);
							\draw[edge] (l1)  to (m1);
						\end{tikzpicture}
						\caption{Extension $vi$-simultaneous $(4,1)$-coloring $c$ to a $vi$-simultaneous $(5,1)$-coloring $c'$.}
						\label{4q+1}
					\end{center}
				\end{figure}
				(ii) $n=4q+2$, $q\in\mathbb{N}$ and $q\in\mathbb{N}$. Figure \ref{C6} shows a $vi$-simultaneous $(5,1)$-coloring of $C_6$. Now suppose that $n\geq 10$. Easily we can use the method of case (i) on two edges $e_1=\{v_{1},v_2\}$ and $e_2=\{v_4,v_5\}$ of $C_{n-2}$ to achieve a $vi$-simultaneous $(5,1)$-coloring of $C_n$.\\
				(iii) $n=4q+3$, $q\in\mathbb{N}$. Figure \ref{C6} shows a $vi$-simultaneous $(5,1)$-coloring of $C_7$. Now suppose that $n\geq 11$. Again we use the method of case (i) on three edges $e_1=\{v_1,v_2\}$ (with change the color of $v_{3}$ to $5$ instead of vertex $v_{n-3}$), $e_2=\{v_4,v_5\}$ and $e_3=\{v_7,v_8\}$ of $C_{n-3}$ to achieve a $vi$-simultaneous $(5,1)$-coloring of $C_n$.
				\begin{figure}[h]
					\begin{center}
						\begin{tikzpicture}[scale=1.0]
							\tikzset{vertex/.style = {shape=circle,draw, line width=1pt, opacity=1.0, inner sep=2pt}}
							\tikzset{vertex1/.style = {shape=circle,draw, fill=black, line width=1pt,opacity=1.0, inner sep=2pt}}
							\tikzset{edge/.style = {-,> = latex', line width=1pt,opacity=1.0}}
							\node[vertex1] (a) at (0,0) {};
							\node  at (0,-0.4) {$1$};
							\node[vertex] (a1) at (1,0) {};
							\node  at (1,-0.4) {$3$};
							\node[vertex] (a2) at  (2,0) {};
							\node  at (2,-0.4) {$4$};
							\node[vertex1] (b) at  (3,0) {};
							\node  at (3,-0.4) {$2$};
							\node[vertex] (b1) at  (4,0) {};
							\node  at (4,-0.4) {$5$};
							\node[vertex] (b2) at  (5,0) {};
							\node  at (5,-0.4) {$3$};
							\node[vertex1] (c) at (6,0) {};
							\node  at (6,-0.4) {$1$};
							\node[vertex] (c1) at (7,0) {};
							\node  at (7,-0.4) {$4$};
							\node[vertex] (c2) at (8,0) {};
							\node  at (8,-0.4) {$5$};
							\node[vertex1] (d) at (8,1) {};
							\node  at (8,1.4) {$2$};
							\node[vertex] (d1) at (7,1) {};
							\node  at (7,1.4) {$3$};
							\node[vertex] (d2) at (6,1) {};
							\node  at (6,1.4) {$4$};
							\node[vertex1] (e) at (5,1) {};
							\node  at (5,1.4) {$1$};
							\node[vertex] (e1) at (4,1) {};
							\node  at (4,1.4) {$5$};
							\node[vertex] (e2) at (3,1) {};
							\node  at (3,1.4) {$3$};
							\node[vertex1] (f) at (2,1) {};
							\node  at (2,1.4) {$2$};
							\node[vertex] (f1) at (1,1) {};
							\node  at (1,1.4) {$4$};
							\node[vertex] (f2) at (0,1) {};
							\node  at (0,1.4) {$5$};
							\draw[edge] (a)  to (a1);
							\draw[edge] (a1)  to (a2);
							\draw[edge] (a2)  to (b);
							\draw[edge] (b)  to (b1);
							\draw[edge] (b1)  to (b2);
							\draw[edge] (b2)  to (c);
							\draw[edge] (c)  to (c1);
							\draw[edge] (c1)  to (c2);
							\draw[edge] (c2)  to (d);
							\draw[edge] (d)  to (d1);
							\draw[edge] (d1)  to (d2);
							\draw[edge] (d2)  to (e);
							\draw[edge] (e)  to (e1);
							\draw[edge] (e1)  to (e2);
							\draw[edge] (e2)  to (f);
							\draw[edge] (f)  to (f1);
							\draw[edge] (f1)  to (f2);
							\draw[edge] (f2)  to (a);
							\node[vertex1] (a) at (0,2) {};
							\node  at (0,2.4) {$5$};
							\node[vertex] (a1) at (1,2) {};
							\node  at (1,2.4) {$1$};
							\node[vertex] (a2) at  (2,2) {};
							\node  at (2,2.4) {$3$};
							\node[vertex1] (b) at  (3,2) {};
							\node  at (3,2.4) {$4$};
							\node[vertex] (b1) at  (4,2) {};
							\node  at (4,2.4) {$2$};
							\node[vertex] (b2) at  (5,2) {};
							\node  at (5,2.4) {$1$};
							\node[vertex1] (c) at (6,2) {};
							\node  at (6,2.4) {$5$};
							\node[vertex] (c1) at (7,2) {};
							\node  at (7,2.4) {$3$};
							\node[vertex] (c2) at (8,2) {};
							\node  at (8,2.4) {$2$};
							\node[vertex1] (x) at (9,2) {};
							\node  at (9,1.6) {$1$};
							\node[vertex] (x1) at (9,3) {};
							\node  at (9,3.4) {$4$};
							\node[vertex] (x2) at (8,3) {};
							\node  at (8,3.4) {$3$};
							\node[vertex1] (d) at (7,3) {};
							\node  at (7,3.4) {$2$};
							\node[vertex] (d1) at (6,3) {};
							\node  at (6,3.4) {$5$};
							\node[vertex] (d2) at (5,3) {};
							\node  at (5,3.4) {$4$};
							\node[vertex1] (e) at (4,3) {};
							\node  at (4,3.4) {$3$};
							\node[vertex] (e1) at (3,3) {};
							\node  at (3,3.4) {$2$};
							\node[vertex] (e2) at (2,3) {};
							\node  at (2,3.4) {$5$};
							\node[vertex1] (f) at (1,3) {};
							\node  at (1,3.4) {$4$};
							\node[vertex] (f1) at (0,3) {};
							\node  at (0,3.4) {$3$};
							\node[vertex] (f2) at (-1,2.5) {};
							\node  at (-1,2.1) {$2$};
							\draw[edge] (a)  to (a1);
							\draw[edge] (a1)  to (a2);
							\draw[edge] (a2)  to (b);
							\draw[edge] (b)  to (b1);
							\draw[edge] (b1)  to (b2);
							\draw[edge] (b2)  to (c);
							\draw[edge] (c)  to (c1);
							\draw[edge] (c1)  to (c2);
							\draw[edge] (c2)  to (x);
							\draw[edge] (x)  to (x1);
							\draw[edge] (x1)  to (x2);
							\draw[edge] (x2)  to (d);
							\draw[edge] (d)  to (d1);
							\draw[edge] (d1)  to (d2);
							\draw[edge] (d2)  to (e);
							\draw[edge] (e)  to (e1);
							\draw[edge] (e1)  to (e2);
							\draw[edge] (e2)  to (f);
							\draw[edge] (f)  to (f1);
							\draw[edge] (f1)  to (f2);
							\draw[edge] (f2)  to (a);
						\end{tikzpicture}
						\caption{$vi$-simultaneous $(5,1)$-coloring $C_6$ and $C_7$.}
						\label{C6}
					\end{center}
				\end{figure}
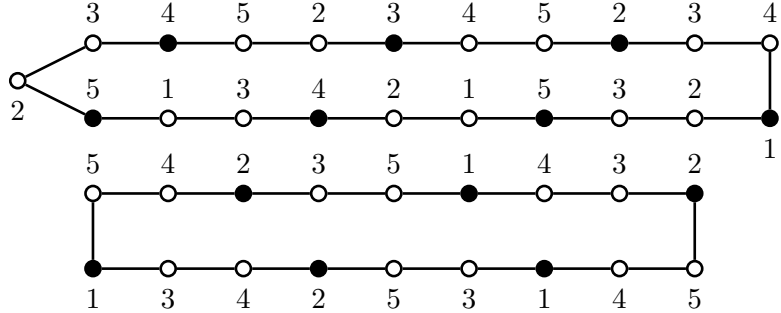
		}\end{proof}
		\begin{corollary}
			Let $G$ be a nonempty graph with $\Delta(G)\leq2$. Then $\chi_{vi,1}(G)=4$ if and only if each component of $G$ is a cycle of order divisible by 4 or a path.
		\end{corollary}
		The following lemma is about the underlying digraph of any subgraph of $\frac{3}{3}$-power of a graph induces by an independence set. We leave the proof to the reader.
		\begin{lemma}\label{stardiforest}
			Let $G$ be a graph and $S$ be an independent set of $G^{\frac{3}{3}}$. Then each component of $D(G^{\frac{3}{3}}[S])$ is trivial or star whose arcs are directed towards the center. In addition the vertices of trivial components form an independent set in $G$.
		\end{lemma}
		\begin{theorem}\label{complete}
			$\chi_{vi}(K_n)=n+2$ for each $n\in\mathbb{N}\setminus\{1\}$.
		\end{theorem}
		\begin{proof}{
				Let $G=K_n^{\frac{3}{3}}$, $c:V(G)\rightarrow [\chi(G)]$ be a proper coloring and $C_j=c^{-1}(j)$ ($1\leq j\leq\chi(G)$). Lemma \ref{stardiforest} concludes that each color class $C_j$ has at most $n-1$ vertices. So
				\[\chi(G)\geq\frac{|V(G)|}{n-1}=\frac{n^2}{n-1}=n+1+\frac{1}{n-1}.\]
				Therefore, $\chi(G)\geq n+2$. Now we define a proper $(n+2)$-coloring of $G$.\\	 
				When $n=2$, $\chi(G)=\chi(K_4)=4$. Now we consider $n\geq 3$. Consider the hamiltonian cycle of $K_n$, named $C=(v_1,v_2,\ldots,v_n)$. For $1\leq j\leq n$,  assign color $j$ to the $t$-vertex $v_j$ and all $i$-vertices $(v_k,v_{j+1})$, where $k\in [n]\setminus\{j,j+1\}$ and $v_{n+1}=v_1$. It can be easily seen that, all $t$-vertices of $G$ have a color in $[n]$ and the only uncolored vertices of $G$ are $(v_j,v_{j+1})$, for $1\leq j\leq n$. Now, it is enough to color the mentioned $i$-vertices. Suppose that $n$ is even. Assign color $n+1$  to the $i$-vertex $(v_j,v_{j+1})$, if $j$ is an odd number, otherwise color it with the color $n+2$. Now suppose that $n$ is an odd integer. Then for $1\leq j\leq n-1$, color the $i$-vertex $(v_j,v_{j+1})$ with color $n+1$, if $j$ is odd and otherwise assign color $n+2$ to it. Also, color the $i$-vertex $(v_n,v_1)$ with color $n$ and recolor the $t$-vertex  $v_n$ with color $n+1$.
		}\end{proof}
		Suppose that $c$ is a $vi$-simultaneous $(n+2)$-coloring of $K_n$. For any vertex $v$, $|c(I_1[v])|=n$ and so $c(I_2(v))|=2$. Therefore $\chi_{vi,2}(K_n)=\chi_{vi}(K_n)=n+2$. In the following theorem, we determine $\chi_{vi,1}(K_n)$.
		\begin{theorem}\label{(vi,1)Kn}
			Let $n\in\mathbb{N}\setminus\{1\}$ and $G$ be a graph of order $n$. Then $\chi_{vi,1}(G)=2n$ if and only if $G\cong K_n$.
		\end{theorem}
		\begin{proof}{Firstly, suppose that $G\cong K_n$. Since $diam(G)=1$, by Definition \ref{Tvi1}, any two vertices $(u,i)$ and $(v,j)$ of $\mathcal{T}_{vi,1}(G)$ are adjacent. So $\chi_{vi,1}(G)=\chi(\mathcal{T}_{vi,1}(G))=|V(\mathcal{T}_{vi,1}(G))|=2n$. Conversely, suppose that $\chi_{vi,1}(G)=2n$. Therefore, $\chi(\mathcal{T}_{vi,1}(G))=2n=|V(\mathcal{T}_{vi,1}(G))|$ which implies that $\mathcal{T}_{vi,1}(G)$ is a complete graph. Now for any two distinct vertices $u$ and $v$ of $G$, the vertices $(u,1)$ and $(v,2)$ of $\mathcal{T}_{vi,1}(G)$ are adjacent and so $d_G(u,v)=1$. Thus $G$ is a complete graph.
		}\end{proof}
		A dynamic coloring of a graph $G$ is a proper coloring, in which each vertex neighborhood of size at least two receives at least two distinct colors. The dynamic chromatic number $\chi_d(G)$ is the least number of colors in such a coloring of $G$ \cite{Dynamic}. Akbari et al. proved the following theorem that we use to give a proper coloring for $\frac{3}{3}$-power of a regular bipartite graph. 
		\begin{theorem} {\em{\cite{Akbari}}}\label{dynamic}
			Let $G$ be a $k$-regular bipartite graph, where $k\geq 4$. Then, there is a $4$-dynamic coloring of $G$, using two colors for each part.
		\end{theorem}
		\begin{theorem} {\em{\cite{bondy}}}\label{Hallregular}
			Every regular bipartite graph has a perfect matching.
		\end{theorem}
		\begin{theorem}\label{regularbipartite}
			If $G=G(A,B)$ is a $k$-regular bipartite graph with $k\geq 4$ and $|A|=|B|=n$, then $\chi_{vi}(G)\leq \min\{n+3,2k\}$.
		\end{theorem}
		\begin{proof}
			{Suppose that $V(A)=\{v_1,\ldots,v_n\}$ and $V(B)=\{u_1,\ldots,u_n\}$. Since $G$ is a $k$-regular bipartite graph, by Theorem~\ref{Hallregular}, $G$ has a perfect matching $M=\{v_1u_1,\ldots,v_nu_n\}$. First, we present a $(n+3)$-proper coloring for $G^{\frac{3}{3}}$.
				For $2\leq i\leq n$ color two $t$-vertices $v_i$ and $u_i$ with colors $1$ and ${n+1}$, respectively.
				Also, for $u\in N(v_1)$ and $v\in N(u_1)$ color $i$-vertices $(u,v_1)$ and $(v,u_1)$ with colors $1$ and $n+1$, respectively.\\
				Now, for $2\leq i\leq n$, for $u\in N(v_i)\setminus\{u_i\}$ and $v\in N(u_i)\setminus\{v_i\}$, assign color $i$ to $i$-vertices $(u,v_i)$ and $(v,u_i)$.
				It can be easily seen that all the $t$-vertices of $G$ except $\{v_1,u_1\}$ and all  $i$-vertices of $G$ except $\{(v_i,u_i),(u_i,v_i)|\hspace{1mm}2\leq i\leq n\}$ have colors in $[n+1]$.
				Now, assign colors  $n+2$ and $n+3$ to $t$-vertices $v_1$ and $v_2$, respectively. Also, for $2\leq i\leq n$, then color  $i$-vertices $(v_i,u_i)$ and $(u_i,v_i)$ with colors $n+2$ and $n+3$, respectively.
				With a simple review, you can see  that this coloring is a proper coloring for $G^{\frac{3}{3}}$ with $(n+3)$ colors.\\
				In the following, we present a $(2k)$-proper coloring for $G^{\frac{3}{3}}$. 
				By Theorem~\ref{dynamic}, there is a $4$-dynamic coloring of $G$, named $c$, using two colors in each part. Without loss of generality, suppose that each $t$-vertex in $A$ has one of colors $1$ and $2$ and each $t$-vertex in  $B$ has one of colors $3$ or $4$.
				For $1\leq i\leq n$, consider the $t$-vertex $u_i\in V(B)$ with set of neighbors $N(u_i)$. Note that, $c$ is a $4$-dynamic coloring, so $u_i$ has at least one neighbor of each colors $1$ and $2$. Let $u$ and $u'$ be two $t$-vertices in $N(u_i)$, where $c(u)=1$ and $c(u')=2$. First, assign colors $1$ and $2$ to  $i$-vertices $(u_i,u')$ and $(u_i,u)$, respectively. Then,  for $w\in N(u_i)\setminus \{u,u'\}$, color all  $i$-vertices $(u_i,w)$ with different colors in $\{5,\ldots,{k+2}\}$.
				Similarly, for a $t$-vertex $v_i\in V(A)$, Suppose that $v$ and $v'$ are neighbors of $v$ with colors $3$ and $4$, respectively.  Color the  $i$-vertices  $(v_i,v')$ and $(v_i,v)$  with colors $3$ and $4$, respectively. Then,  for $w'\in N(v_i)\setminus \{v,v'\}$, color all  $i$-vertices $(v_i,w')$ with different colors in $\{k+3,\ldots,2k\}$. It can be easily seen that, the presented coloring is a proper $(2k)$-coloring for $G^{\frac{3}{3}}$.
		}\end{proof}
		Since any bipartite graph with maximum degree $\Delta$ can be extended to a $\Delta$-regular bipartite graph, we have the following corollary.
		\begin{corollary}
			If $G$ is a bipartite graph with maximum degree $\Delta$, then $\chi_{vi}(G)\leq 2\Delta$.
		\end{corollary}
		A derangement of a set $S$ is a bijection $\pi : S\rightarrow S$ such that no element $x\in S$ has $\pi(x)=x$.
		\begin{theorem}
			Let  $n,m\in\mathbb{N}$ and $n\geq m$. Then  $\chi_{vi}(K_{n,m})=\left\{\begin{array}{ll} n+2 & m\leq 2\\ n+3 & m\geq 3\end{array}\right.$.
		\end{theorem}
		\begin{proof}{
				Let $A=\{v_1,\ldots,v_n\}$ and $B=\{u_1,\ldots,u_m\}$ be two parts of $K_{n,m}$ and $G=K_{n,m}^{\frac{3}{3}}$. If $m=1$, then $K_{n,1}$ is a tree and by Corollary~\ref{tree}, we have $\chi(G)=n+2$. Now suppose that $m=2$. Since $\omega(G)=\Delta+2$, $\chi(G)\geq n+2$. It suffices to present a proper $(n+2)$-coloring for $G$ with colors in $[n+2]$. Suppose that $\pi$ is a derangement of the set $[n]$. Assign color $n+1$ to the vertices of $\{u_1\}\cup I_2(u_2)$ and color $n+2$ to the vertices of $u_2\cup I_2(u_1)$. Also for $j\in[n]$, color $i$-vertices $(u_1,v_j)$ and $(u_2,v_j)$ with color $j$ and vertex $v_j$ with color $\pi(j)$. The given coloring is a proper $(n+2)$-coloring of  $G$.\\
				In the case $m\geq 3$, suppose that $c$ is a proper coloring of $G$ with colors $1,\ldots,n+2$. Since the vertices of $I_1[u_1]$ are pairwise adjacent in $G$, there are exactly $n+1$ colors  in $c(I_1[u_1])$. Without loss of generality, suppose that $c(u_1)=1$ and $c(I_1(u_1))=[n+1]\setminus\{1\}$. By Theorem~\ref{firstlem}, all $i$-vertices of $I_2(u_1)$ have the same color $n+2$.\\
				Now, consider $t$-vertices $u_2$ and $u_3$. All $i$-vertices of $I_2(u_2)$ and all $i$-vertices of $I_2(u_3)$, have the same color and their colors are different from $\{2,\ldots,n+2\}$. Hence, the only available color for these vertices is the color $1$. But the subgraph of $G$ induced by $I_2(u_2)\cup I_2(u_3)$ is 1-regular and so for their coloring we need to two colors, a contradiction.\\
				To complete the proof, it suffices to show that $\chi((K_{n,n})^{\frac{3}{3}})\leq n+3$. Since $n\geq 3$, $n+3\leq 2n$ and by Theorem~\ref{regularbipartite}, we have $\chi(G)\leq\chi({K_{n,n}}^{\frac{3}{3}})\leq \min\{n+3,2n\}=n+3$. Hence, $\chi(G)=n+3$.
		}\end{proof}
		\begin{theorem}\label{vi1Knm}
			Let $n,m\in\mathbb{N}\setminus\{1\}$. Then $\chi_{vi,1}(K_{n,m})=n+m$.
		\end{theorem}
		\begin{proof}{
				Since $(K_{n,m})^2\cong K_{n+m}$, $K_{n+m}$ is a subgraph of $\mathcal{T}_{vi,1}(K_{n,m})$ and so $\chi_{vi.1}(K_{n,m})=\chi(\mathcal{T}_{vi,1}(K_{n,m}))\geq n+m$. Now we show that $\chi(\mathcal{T}_{vi,1}(K_{n,m}))\leq n+m$. Let $V=\{v_1,\ldots,v_n\}$ and $U=\{u_1,\ldots,u_m\}$ be two parts of $K_{n,m}$, $\pi$ be a derangement of $[n]$ and $\sigma$ be a derangement of $[m]$. Easily one can show that the following vertex coloring of $\mathcal{T}_{vi,1}(K_{n,m})$ is proper.
				\[c(x)=\left\{\begin{array}{llll} i & x=(v_i,2)\\ n+j & x=(u_j,2)\\ \pi(i) & x=(v_i,1)\\ n+\sigma(j) & x=(u_j,1).\end{array}\right.\]
		}\end{proof}
		As we see, there are some graphs such as trees and $K_{n,2}$ with maximum degree $\Delta$, whose $\frac{3}{3}$-power has chromatic number equal to $\Delta+2$. The problem is Characterizing all graphs with the desired property.
		\begin{problem}{\rm
				Charactrize all graphs $G$ with maximum degree at least 3 such that $\chi(G^{\frac{3}{3}})=\omega(G^{\frac{3}{3}})=\Delta(G)+2$.
		}\end{problem}
		{\bf Acknowledgements.} This research was in part supported by a grant from IPM (No.1400050116).

	\end{document}